\documentclass{amsart}
\usepackage{amssymb,latexsym,amsmath,amscd,graphicx,graphics,epic,eepic,bm,color,array,mathrsfs,fullpage}
\usepackage{enumerate}
\usepackage[all,knot,poly]{xy}

\newcommand{\R}{\mathbb{R}}

\newcommand{\lm}{\lambda}
\newcommand{\s}{\sigma}

\newcommand{\vf}{\varphi}

\newcommand{\Tr}{\mathrm{Tr}}

\theoremstyle{plain}
\newtheorem{theorem}{Theorem}[section]
\newtheorem{lemma}[theorem]{Lemma}
\newtheorem{proposition}[theorem]{Proposition}
\newtheorem{corollary}[theorem]{Corollary}

\theoremstyle{definition}
\newtheorem{definition}[theorem]{Definition}

\theoremstyle{remark}

\numberwithin{equation}{section}

\begin{document}

\title{A Stochastic Gradient Descent Theorem and the Back-Propagation Algorithm}

\author{Hao Wu}

\thanks{}

\address{Department of Mathematics, The George Washington University, Phillips Hall , Room 739, 801 22nd Street, N.W., Washington DC 20052, USA. Telephone: 1-202-994-0653, Fax: 1-202-994-6760}

\email{haowu@gwu.edu}



\begin{abstract}
We establish a convergence theorem for a certain type of stochastic gradient descent, which leads to a convergent variant of the back-propagation algorithm.
\end{abstract}

\maketitle

\section{A Convergence Theorem for Stochastic Gradient Descent}\label{sec-sgd}

\subsection{Statement of the theorem}

We present elements of $R^l$ by column vectors and denote by $\|\mathbf{z}\|=\sqrt{z^T z}$ the standard norm on $\R^l$. For any $0<r\leq \infty$ and $\mathbf{z}\in \R^l$, we denote by $B^l_r(\mathbf{z})$ the open ball of radius $r$ centered on $\mathbf{z}$ in $\R^l$. In particular, $B^l_\infty(\mathbf{z})=\R^l$.

\begin{theorem}\label{thm-sgd}
Fix a $0 <\rho\leq \infty$. Assume that:
\begin{enumerate}
	\item $\mu$ is a probability measure on $B^m_\rho(\mathbf{0})$.
	\item\label{assumption-f} $f:\R^n \times B^m_\rho(\mathbf{0}) \rightarrow \R$ is a function satisfying:
	\begin{enumerate}
		\item For any $\mathbf{x} \in \R^n$, the function $f(\mathbf{x},\ast)$ is $\mu$-measurable on $B^m_\rho(\mathbf{0})$.
		\item For any $\mathbf{y} \in B^m_\rho(\mathbf{0})$, the function $f(\ast,\mathbf{y})$ has continuous firs-order partial derivatives over $\R^n$.
		\item\label{assumption-F-bounded} For any $r>0$,
		\begin{itemize}
		  \item the function $f(\mathbf{x},\mathbf{y})$ is bounded on $B^n_r(\mathbf{0}) \times B^m_\rho(\mathbf{0})$,
			\item $\|\nabla_\mathbf{x} f(\mathbf{x},\mathbf{y})\|$ is bounded on $B^n_r(\mathbf{0}) \times B^m_\rho(\mathbf{0})$, where $\nabla_\mathbf{x} f(\mathbf{x},\mathbf{y})$ is the gradient vector of $f(\mathbf{x},\mathbf{y})$ with respect to $\mathbf{x} \in \R^n$,
			\item $\|\nabla_\mathbf{x} f(\mathbf{x},\mathbf{y})\|$ is Lipschitz with respect to $\mathbf{x}$ on $B^n_r(\mathbf{0}) \times B^m_\rho(\mathbf{0})$. That is, there is an $M>0$ such that $\|\nabla_\mathbf{x} f(\mathbf{x},\mathbf{y})-\nabla_\mathbf{x} f(\mathbf{x}',\mathbf{y})\|\leq M\|\mathbf{x}-\mathbf{x}'\|$ for any $\mathbf{x},\mathbf{x}'\in B^n_r(\mathbf{0})$ and $\mathbf{y}\in B^m_\rho(\mathbf{0})$.
		\end{itemize}
		\item\label{assumption-gradient} There is an $R_0>0$ such that $\mathbf{x}^T \nabla_\mathbf{x} f(\mathbf{x},\mathbf{y}) \geq 0$ for all $\mathbf{x} \geq R_0$ and $\mathbf{y} \in B^m_\rho(\mathbf{0})$.
	\end{enumerate}
	\item $\{a_k\}_{k=0}^\infty$ is a sequence of positive numbers satisfying $\sum_{k=0}^\infty a_k =\infty$ and $\sum_{k=0}^\infty a_k^2 <\infty$. Define $A = \sup\{a_k~|~k\geq 0\}$.
	\item 
	\begin{itemize}
		\item $\mathbf{x}_0$ is any fixed element of $\R^n$, 
		\item $R_1 = \max\{\sqrt{\|\mathbf{x}_0\|^2 + \sum_{k=0}^\infty a_k^2},\sqrt{R_0^2 +2AR_0 +\sum_{k=0}^\infty a_k^2}\}$,
		\item $\Phi =\sup \{\|\nabla_\mathbf{x} f(\mathbf{x},\mathbf{y})\|~|~ (\mathbf{x},\mathbf{y}) \in B^n_{R_1}(\mathbf{0}) \times B^m_\rho(\mathbf{0})\}$.
	\end{itemize}
	\item $\{\mathbf{y}_k\}_{k=0}^\infty$ is a sequence of independent random variables with identical probability distribution taking values in $B^m_\rho(\mathbf{0})$, where the probability distribution of each $\mathbf{y}_k$ is $\mu$.
	\item For any $\mathbf{x} \in \R^n$, $F(\mathbf{x}) = \int_{B^m_\rho(\mathbf{0})} f(\mathbf{x},\ast) d\mu$.
\end{enumerate}
Fix a positive number $\vf$ satisfying $\vf\geq \Phi$. Define a sequence $\{\mathbf{x}_k\}_{k=0}^\infty$ of random variables by 
\begin{equation}\label{eq-def-x-k}
\mathbf{x}_{k+1} = \mathbf{x}_k - \frac{a_k}{\vf} \nabla_\mathbf{x} f(\mathbf{x}_k,\mathbf{y}_k) \text{ for } k\geq 0.
\end{equation}
Then 
\begin{itemize}
	\item $\{F(\mathbf{x}_k)\}$ converges almost surely to a finite number,
	\item $\{\nabla F(\mathbf{x}_k)\}$ converges almost surely to $\mathbf{0}$,
	\item any limit point of $\{\mathbf{x}_k\}$ is almost surely a stationary point of the function $F$.
\end{itemize}
\end{theorem}

Before moving onto the proof of Theorem \ref{thm-sgd}, it is worth pointing out that the most restrictive assumption about the function $f$ in this theorem is Assumption (\ref{assumption-gradient}). All the other assumptions about $f$ are fairly non-restrictive. Assumption (\ref{assumption-gradient}) was essentially used by Bottou in \cite[Section 5]{Bottou}. However, to deduce the convergence of the stochastic gradient descent, he imposed additional restrictive assumptions in \cite[Subsection 5.1]{Bottou}. In the definition of the sequence $\{\mathbf{x}_k\}$ above, we divide the learning rate $a_k$ by the (potentially large) constant $\vf$, which eliminates the need for Bottou's additional assumptions at the cost of having smaller steps in the gradient descent. So, comparing to Bottou's results, Theorem \ref{thm-sgd} applies to a more general class of functions but leads to a potentially slower convergence.

\subsection{Proof of the theorem} The key step in the proof Theorem \ref{thm-sgd} is to prove that $\{\mathbf{x}_k\}$ is bounded, which turns out to be fairly simple. The rest of the proof is somewhat more technical, but mainly consists of arguments similar to those already appeared in literature on stochastic approximation such as \cite{Bertsekas-Tsitsiklis-gradient,Blum-miltidimensional-stochastic-approximation,Robbins-Monro}.

\begin{lemma}\label{lemma-x-bounded}
$\|\mathbf{x}_k\| < R_1$ for all $k\geq 0$.
\end{lemma}
\begin{proof}
We claim that $\|\mathbf{x}_k\|^2 + \sum_{j=k}^\infty a_j^2 \leq R_1^2$, which implies the lemma. We prove this claim by an induction on $k$. The claim is true for $k=0$ by the definition of $R_1$. Assume that the claim is true for $k$. Consider two cases:
\begin{itemize}
	\item[Case 1:] $\|\mathbf{x}_k\| \leq R_0$. In this case, we have that
	\begin{eqnarray*}
	&& \|\mathbf{x}_{k+1}\|^2 = \|\mathbf{x}_k - \frac{a_k}{\vf} \nabla_\mathbf{x} f(\mathbf{x}_k,\mathbf{y}_k)\|^2 = \|\mathbf{x}_k\|^2 -2 \frac{a_k}{\vf} \mathbf{x}_k^T \nabla_\mathbf{x} f(\mathbf{x}_k,\mathbf{y}_k) + \frac{a_k^2}{\vf^2}\|\nabla_\mathbf{x} f(\mathbf{x}_k,\mathbf{y}_k)\|^2 \\
	& \leq & \|\mathbf{x}_k\|^2 +2 a_k \|\mathbf{x}_k\| \frac{\|\nabla_\mathbf{x} f(\mathbf{x}_k,\mathbf{y}_k)\|}{\vf} + a_k^2\frac{\|\nabla_\mathbf{x} f(\mathbf{x}_k,\mathbf{y}_k)\|^2}{\vf^2} \leq R_0^2 + 2A R_0 + a_k^2.
	\end{eqnarray*}
	So $\|\mathbf{x}_{k+1}\|^2 + \sum_{j=k+1}^\infty a_j^2 \leq R_0^2 + 2A R_0 + \sum_{j=k}^\infty a_j^2 \leq R_1^2$.
	\item[Case 2:] $R_0< \|\mathbf{x}_k\| < R_1$. In this case, using Assumption (\ref{assumption-gradient}), we have that
	\[ 
	\|\mathbf{x}_{k+1}\|^2 =  \|\mathbf{x}_k\|^2 -2 \frac{a_k}{\vf} \mathbf{x}_k^T \nabla_\mathbf{x} f(\mathbf{x}_k,\mathbf{y}_k) + \frac{a_k^2}{\vf^2}\|\nabla_\mathbf{x} f(\mathbf{x}_k,\mathbf{y}_k)\|^2 \leq \|\mathbf{x}_k\|^2 + a_k^2.
	\]
	So $\|\mathbf{x}_{k+1}\|^2 + \sum_{j=k+1}^\infty a_j^2 \leq \|\mathbf{x}_k\|^2 + \sum_{j=k}^\infty a_j^2 \leq R_1^2$.
\end{itemize}
Combining these two cases, one can see that the claim is true for $k+1$. This completes the induction and proves the claim. And the lemma follows.
\end{proof}

\begin{lemma}\label{lemma-lipschitz}
$\nabla F$ is Lipschitz on $B^n_{R_1}(\mathbf{0})$. For any $\mathbf{x},\mathbf{x}' \in B^n_{R_1}(\mathbf{0})$, we have $\|\nabla F(\mathbf{x}') -\nabla F(\mathbf{x})\| \leq M \|\mathbf{x}' -\mathbf{x}\|$, where $M$ comes from Condition (\ref{assumption-F-bounded}) in Theorem \ref{thm-sgd}. Moreover, for such $\mathbf{x}$ and $\mathbf{x}'$, we have $F(\mathbf{x}') \leq F(\mathbf{x}) + (\mathbf{x}'-\mathbf{x})^T\nabla F(\mathbf{x}) + \frac{M}{2}\|\mathbf{x}'-\mathbf{x}\|^2$.
\end{lemma}

\begin{proof}
By the Dominated Convergence Theorem, we have that the gradient of $F$ is $\nabla F(\mathbf{x}) = \int_{B^m_\rho(\mathbf{0})} \nabla_\mathbf{x} f(\mathbf{x},\ast) d\mu$. So, for any $\mathbf{x},\mathbf{x}' \in B^n_{R_1}(\mathbf{0})$, we have 
\begin{eqnarray*}
 \|\nabla F(\mathbf{x}') -\nabla F(\mathbf{x})\| & = & \|\int_{B^m_\rho(\mathbf{0})} (\nabla_\mathbf{x} f(\mathbf{x},\ast)-\nabla_\mathbf{x} f(\mathbf{x}',\ast)) d\mu\| \leq  \int_{B^m_\rho(\mathbf{0})} \|\nabla_\mathbf{x} f(\mathbf{x},\ast)-\nabla_\mathbf{x} f(\mathbf{x}',\ast)\| d\mu \\
& \leq & \int_{B^m_\rho(\mathbf{0})} M\|\mathbf{x} -\mathbf{x}'\| d\mu = M\|\mathbf{x} -\mathbf{x}'\|.
\end{eqnarray*}
Next, for any $\mathbf{x},\mathbf{x}' \in B^n_{R_1}(\mathbf{0})$, define $g(t) := F((1-t) \mathbf{x} +t\mathbf{x}')$ for $0\leq t \leq 1$. Then $g'(t) = (\mathbf{x}'-\mathbf{x})^T \nabla F((1-t) \mathbf{x} +t\mathbf{x}')$. And 
\begin{eqnarray*}
F(\mathbf{x}') - F(\mathbf{x}) & = & g(1) -g(0) = \int_0^1 g'(t)dt = \int_0^1 (\mathbf{x}'-\mathbf{x})^T \nabla F((1-t) \mathbf{x} +t\mathbf{x}')dt \\
& = & \int_0^1 (\mathbf{x}'-\mathbf{x})^T (\nabla F(\mathbf{x}) + \nabla F((1-t) \mathbf{x} +t\mathbf{x}')-\nabla F(\mathbf{x})) dt \\
& = & \int_0^1 (\mathbf{x}'-\mathbf{x})^T \nabla F(\mathbf{x}) dt + \int_0^1 (\mathbf{x}'-\mathbf{x})^T (\nabla F((1-t) \mathbf{x} +t\mathbf{x}')-\nabla F(\mathbf{x})) dt \\
& \leq  & (\mathbf{x}'-\mathbf{x})^T \nabla F(\mathbf{x}) + \int_0^1 \|\mathbf{x}'-\mathbf{x}\| \cdot \|\nabla F((1-t) \mathbf{x} +t\mathbf{x}')-\nabla F(\mathbf{x})\| dt \\
& \leq  & (\mathbf{x}'-\mathbf{x})^T \nabla F(\mathbf{x}) + \int_0^1 M\|\mathbf{x}'-\mathbf{x}\| \cdot \|(1-t) \mathbf{x} +t\mathbf{x}'-\mathbf{x}\| dt \\
& = & (\mathbf{x}'-\mathbf{x})^T \nabla F(\mathbf{x}) + M\|\mathbf{x}'-\mathbf{x}\|^2\int_0^1 t dt = (\mathbf{x}'-\mathbf{x})^T \nabla F(\mathbf{x}) + \frac{M}{2}\|\mathbf{x}'-\mathbf{x}\|^2
\end{eqnarray*}
This proves the lemma.
\end{proof}

\begin{lemma}\label{lemma-gradient-square-sum-converge}
$\sum_{k=0}^\infty a_k E(\|\nabla F(\mathbf{x}_k)\|^2)$ converges and, consequently, $\sum_{k=0}^\infty a_k \|\nabla F(\mathbf{x}_k)\|^2$ converges almost surely.
\end{lemma}
\begin{proof}
By Lemma \ref{lemma-x-bounded}, we know that $\mathbf{x}_k < R_1$ and therefore $\|\nabla_\mathbf{x} f(\mathbf{x}_k,\mathbf{y}_k)\| \leq \vf$. Let $M$ be as in Lemma \ref{lemma-lipschitz}. We have that
\begin{equation}\label{eq-inequality-F-induction}
F(\mathbf{x}_{k+1}) \leq F(\mathbf{x}_{k}) - \frac{a_k}{\vf} \nabla F(\mathbf{x}_k)^T\nabla_\mathbf{x} f(\mathbf{x}_k,\mathbf{y}_k) + \frac{a_k^2}{2} M.
\end{equation}
Taking the expectation on both sides of Inequality \eqref{eq-inequality-F-induction}, we get 
\[
E(F(\mathbf{x}_{k+1})) \leq E(F(\mathbf{x}_{k})) - \frac{a_k}{\vf} E(\nabla F(\mathbf{x}_k)^T\nabla_\mathbf{x} f(\mathbf{x}_k,\mathbf{y}_k)) + \frac{a_k^2}{2} M.
\]
But 
\[
E(\nabla F(\mathbf{x}_k)^T\nabla_\mathbf{x} f(\mathbf{x}_k,\mathbf{y}_k)) = E(E(\nabla F(\mathbf{x}_k)^T\nabla_\mathbf{x} f(\mathbf{x}_k,\mathbf{y}_k)~|~\mathbf{x}_k)) = E(\|\nabla F(\mathbf{x}_k)\|^2). 
\]
So 
\begin{equation}\label{eq-inequality-F-exp-induction}
E(F(\mathbf{x}_{k+1})) \leq E(F(\mathbf{x}_{k})) - \frac{a_k}{\vf} E(\|\nabla F(\mathbf{x}_k)\|^2) + \frac{a_k^2}{2} M.
\end{equation}
Summing Inequality \eqref{eq-inequality-F-exp-induction} from $0$ to $k$, we get that
\begin{equation}\label{eq-inequality-F-exp-induction-sum}
E(F(\mathbf{x}_{k+1})) \leq F(\mathbf{x}_{0}) - \frac{1}{\vf}\sum_{j=0}^k  a_j E(\|\nabla F(\mathbf{x}_j)\|^2)+ \frac{M}{2}\sum_{j=0}^k a_j^2.
\end{equation}
By Lemma \ref{lemma-x-bounded} and Assumption (\ref{assumption-F-bounded}), there is a $B>0$ such that $|F(\mathbf{x}_{k})| \leq B$ for all $k\geq 0$. Then Inequality \eqref{eq-inequality-F-exp-induction-sum} implies that
\[
\sum_{j=0}^k  a_j E(\|\nabla F(\mathbf{x}_j)\|^2) \leq \vf\left(B + F(\mathbf{x}_{0}) + \frac{M}{2}\sum_{j=0}^k a_j^2\right).
\]
Since $\sum_{k=0}^\infty a_k^2$ converges, this inequality shows that $\sum_{k=0}^\infty a_k E(\|\nabla F(\mathbf{x}_k)\|^2)$ converges. It then follows that $\sum_{k=0}^\infty a_k \|\nabla F(\mathbf{x}_k)\|^2$ is finite with probability $1$. In other words, $\sum_{k=0}^\infty a_k \|\nabla F(\mathbf{x}_k)\|^2$ converges almost surely.
\end{proof}

\begin{lemma}\label{lemma-gradient-mixed-square-sum-converge}
Both $\sum_{k=0}^\infty a_k \nabla F(\mathbf{x}_k)^T\nabla_\mathbf{x} f(\mathbf{x}_k,\mathbf{y}_k)$ and $\lim_{k\rightarrow\infty} F(\mathbf{x}_k)$ converge almost surely.
\end{lemma}
\begin{proof}
Set 
\begin{equation}\label{eq-def-u}
u_k = \nabla F(\mathbf{x}_k)^T(\nabla_\mathbf{x} f(\mathbf{x}_k,\mathbf{y}_k)-\nabla F(\mathbf{x}_k)).
\end{equation}
Consider the sequence $\{z_k=\sum_{j=0}^k a_j u_j\}$ of random variables. By Lemma \ref{lemma-x-bounded} and Assumption (\ref{assumption-F-bounded}), there is a constant $U$ such that $|u_k|\leq U$ for all $k$. Note that each $\mathbf{x}_k$ is determined by $\mathbf{y}_0,\dots, \mathbf{y}_{k-1}$ and is independent of $\mathbf{y}_k$ and that $E(\nabla_\mathbf{x} f(\mathbf{x}_k,\mathbf{y}_k)~|~ \mathbf{x}_k) = \nabla F(\mathbf{x}_k)$. So $E(u_k~|~\mathbf{y}_0,\dots, \mathbf{y}_{k-1})=0$ and $E(z_k~|~\mathbf{y}_0,\dots, \mathbf{y}_{k-1}) = z_{k-1}$. This shows that $\{z_k\}$ is a Martingale relative to $\{\mathbf{y}_k\}$. For any $k\geq 0$, we have that
\[
E(z_k)=E(z_{k-1})+a_kE(u_k) = E(z_{k-1})+a_k E(E(u_k~|~\mathbf{y}_0,\dots, \mathbf{y}_{k-1})) = E(z_{k-1}) =\cdots = E(z_{0})=0.
\]
So the variance of $z_k$ satisfies
\begin{eqnarray*}
Var(z_k) & = & E(z_k^2)= E((\sum_{j=0}^k a_j u_j)^2) = E(a_k^2u_k^2+ (\sum_{j=0}^{k-1} a_j u_j)^2 + 2a_k u_k(\sum_{j=0}^{k-1} a_j u_j))\\
& = & E(a_k^2u_k^2) + Var(z_{k-1}) +2a_k E(u_k(\sum_{j=0}^{k-1} a_j u_j))  \\
& = & E(a_k^2u_k^2) + Var(z_{k-1}) +2a_k E(E(u_k(\sum_{j=0}^{k-1} a_j u_j)~|~\mathbf{y}_0,\dots, \mathbf{y}_{k-1})) \\
& = & E(a_k^2u_k^2) + Var(z_{k-1}) \leq a_k^2 U^2 + Var(z_{k-1})
\end{eqnarray*}
Summing from $0$ to $k$, we get $Var(z_k) \leq U^2\sum_{j=0}^k a_j^2$ for all $k$. But $\sum_{k=0}^\infty a_k^2$ converges. So $\{Var(z_k)\}$ is a bounded sequence. Thus, by the Martingale Convergence Theorem, $\sum_{k=0}^\infty a_k u_k = \lim_{k\rightarrow\infty} z_k$ converges almost surely. By Lemma \ref{lemma-gradient-square-sum-converge}, $\sum_{k=0}^\infty a_k \|\nabla F(\mathbf{x}_k)\|^2$ also converges almost surely. Therefore, 
\[
\sum_{k=0}^\infty a_k \nabla F(\mathbf{x}_k)^T\nabla_\mathbf{x} f(\mathbf{x}_k,\mathbf{y}_k) = \sum_{k=0}^\infty a_k u_k + \sum_{k=0}^\infty a_k \|\nabla F(\mathbf{x}_k)\|^2
\] 
converges almost surely.

Next we consider $\lim_{k\rightarrow\infty} F(\mathbf{x}_k)$. Assume that $\sum_{k=0}^\infty a_k \nabla F(\mathbf{x}_k)^T\nabla_\mathbf{x} f(\mathbf{x}_k,\mathbf{y}_k)$ converges. By Inequality \eqref{eq-inequality-F-induction}, one can see that
\[
\{v_k := F(\mathbf{x}_{k}) - \frac{1}{\vf} \sum_{j=k}^\infty a_j \nabla F(\mathbf{x}_j)^T\nabla_\mathbf{x} f(\mathbf{x}_j,\mathbf{y}_j) + \frac{M}{2} \sum_{j=k}^\infty a_j^2\}
\]
is a decreasing sequence.  By Lemma \ref{lemma-x-bounded} and Assumption (\ref{assumption-F-bounded}), $\{F(\mathbf{x}_{k})\}$ is a bounded sequence. Since $\sum_{k=0}^\infty a_k \nabla F(\mathbf{x}_k)^T\nabla_\mathbf{x} f(\mathbf{x}_k,\mathbf{y}_k)$ and $\sum_{k=0}^\infty a_k^2$ both converge, $\{\sum_{j=k}^\infty a_j \nabla F(\mathbf{x}_j)^T\nabla_\mathbf{x} f(\mathbf{x}_j,\mathbf{y}_j)\}$ and $\{\sum_{j=k}^\infty a_j^2\}$ are also bounded sequences. This shows that $\{v_k\}$ is a bounded decreasing sequence. Therefore, $\lim_{k\rightarrow\infty} v_k$ converges. But $\lim_{k\rightarrow\infty} \sum_{j=k}^\infty a_j \nabla F(\mathbf{x}_j)^T\nabla_\mathbf{x} f(\mathbf{x}_j,\mathbf{y}_j) = \lim_{k\rightarrow\infty} \sum_{j=k}^\infty a_j^2 =0$. So $\lim_{k\rightarrow\infty} F(\mathbf{x}_k) = \lim_{k\rightarrow\infty} v_k$ converges. This proves that $\lim_{k\rightarrow\infty} F(\mathbf{x}_k)$ converges if $\sum_{k=0}^\infty a_k \nabla F(\mathbf{x}_k)^T\nabla_\mathbf{x} f(\mathbf{x}_k,\mathbf{y}_k)$ converges. But the latter converges almost surely. So does the former.
\end{proof}

\begin{lemma}\label{lemma-nabla-F-converge}
$\lim_{k\rightarrow\infty} \nabla F(\mathbf{x}_k) =0$ almost surely.
\end{lemma}
\begin{proof}
By Lemmas \ref{lemma-gradient-square-sum-converge} and \ref{lemma-gradient-mixed-square-sum-converge}, $\sum_{k=0}^\infty a_k \|\nabla F(\mathbf{x}_k)\|^2$, $\sum_{k=0}^\infty a_k \nabla F(\mathbf{x}_k)^T\nabla_\mathbf{x} f(\mathbf{x}_k,\mathbf{y}_k)$ and $\lim_{k\rightarrow\infty} F(\mathbf{x}_k)$ all converge almost surely. So, to prove the current lemma, we only need to show that $\lim_{k\rightarrow\infty} \nabla F(\mathbf{x}_k) =0$ if $\sum_{k=0}^\infty a_k \|\nabla F(\mathbf{x}_k)\|^2$, $\sum_{k=0}^\infty a_k \nabla F(\mathbf{x}_k)^T\nabla_\mathbf{x} f(\mathbf{x}_k,\mathbf{y}_k)$ and $\lim_{k\rightarrow\infty} F(\mathbf{x}_k)$ all converge. In the rest of this proof, we assume that the latter three all converge.

First, we claim that $\liminf_{k\rightarrow\infty} \|\nabla F(\mathbf{x}_k)\| =0$. Otherwise, there would be a $t>0$ and $\kappa >0$ such that $\|\nabla F(\mathbf{x}_k)\| \geq t$ if $k \geq \kappa$. Then $\sum_{k=0}^\infty a_k \|\nabla F(\mathbf{x}_k)\|^2 \geq \sum_{k=\kappa}^\infty a_k \|\nabla F(\mathbf{x}_k)\|^2\geq t^2\sum_{k=\kappa}^\infty a_k =\infty$ since $\sum_{k=0}^\infty a_k =\infty$. This is a contradiction. Thus, $\liminf_{k\rightarrow\infty} \|\nabla F(\mathbf{x}_k)\| =0$.

Next, we claim that $\limsup_{k\rightarrow\infty} \|\nabla F(\mathbf{x}_k)\| =0$, too. We prove this claim by contradiction. Let us assume that $\limsup_{k\rightarrow\infty} \|\nabla F(\mathbf{x}_k)\| =s>0$. Let $M$ be as in Lemma \ref{lemma-lipschitz}. Recall that  $\Phi= \sup\{\|\nabla_\mathbf{x} f(\mathbf{x},\mathbf{y})\|~|~ \|\mathbf{x}\|<R_1, \|\mathbf{y}\|<\rho\}$ and $\vf \geq \Phi$. Since $\lim_{k\rightarrow\infty} a_k =0$, there is a $K>0$ such that $a_k M\leq \frac{s}{8}$ for $k>K$. Since $\liminf_{k\rightarrow\infty} \|\nabla F(\mathbf{x}_k)\| =0$ and $\limsup_{k\rightarrow\infty} \|\nabla F(\mathbf{x}_k)\| =s>0$, there exist two infinite sequences $\{p_i\}$ and $\{q_i\}$ of positive integers such that
\begin{itemize}
	\item $K<p_1<q_1<p_2<q_2\cdots<q_{i-1}<p_i<q_i<p_{i+1}<\cdots$,
	\item for $i=1,2\dots$, $\|\nabla f(\mathbf{x}_{p_i})\|<\frac{s}{4}$, $\|\nabla f(\mathbf{x}_{q_i})\|>\frac{s}{2}$ and  $\frac{s}{4}\leq\|\nabla f(\mathbf{x}_k)\|\leq \frac{s}{2}$ for $p_i<k<q_i$.
\end{itemize}
By Lemmas \ref{lemma-x-bounded} and \ref{lemma-lipschitz}, we have that, for all $i\geq 1$, 
\begin{eqnarray*}
\frac{s}{4} & < & \|\nabla f(\mathbf{x}_{q_i})\| - \|\nabla f(\mathbf{x}_{p_i})\| \leq \|\nabla f(\mathbf{x}_{q_i}) - \nabla f(\mathbf{x}_{p_i})\| \leq M\|\mathbf{x}_{q_i} - \mathbf{x}_{p_i}\| = M \|\sum_{k=p_i}^{q_i-1} \frac{a_k}{\vf} \nabla_\mathbf{x} f(\mathbf{x}_k, \mathbf{y}_k)\| \\
& \leq & M \sum_{k=p_i}^{q_i-1} \frac{a_k}{\vf} \|\nabla_\mathbf{x} f(\mathbf{x}_k, \mathbf{y}_k)\| \leq \frac{M\Phi}{\vf} \sum_{k=p_i}^{q_i-1} a_k \leq M \sum_{k=p_i}^{q_i-1} a_k.
\end{eqnarray*}
This shows that 
\begin{equation}\label{eq-sum-a-k-bound}
\sum_{k=p_i}^{q_i-1} a_k \geq \frac{s}{4M} \text{ for all } i. 
\end{equation}
Using Lemmas \ref{lemma-x-bounded} and \ref{lemma-lipschitz} again, we get that, for each $p_i$,
\begin{eqnarray*}
\frac{s}{4} - \|\nabla F(\mathbf{x}_{p_i})\| & \leq & \|\nabla F(\mathbf{x}_{p_i+1})\|-\|\nabla F(\mathbf{x}_{p_i})\| \leq \|\nabla F(\mathbf{x}_{p_i+1})-\nabla F(\mathbf{x}_{p_i})\| \leq M\|\mathbf{x}_{p_i+1} - \mathbf{x}_{p_i}\| \\
& = & \frac{a_{p_i}M}{\vf}\|\nabla_\mathbf{x} f(\mathbf{x}_{p_i}, \mathbf{y}_{p_i})\| \leq \frac{a_{p_i}M\Phi}{\vf} \leq \frac{s}{8}
\end{eqnarray*}
since $p_i>K$. Thus, $\|\nabla F(\mathbf{x}_{p_i})\|\geq \frac{s}{8}$. By Inequality \eqref{eq-inequality-F-induction}, we have
\begin{eqnarray*}
F(\mathbf{x}_{q_i}) & \leq & F(\mathbf{x}_{p_i}) -\sum_{k=p_i}^{q_i-1}\frac{a_k}{\vf} u_k - \sum_{k=p_i}^{q_i-1}\frac{a_k}{\vf} \|\nabla F(\mathbf{x}_{k})\|^2 + \sum_{k=p_i}^{q_i-1} \frac{a_k^2}{2}M, \\
& \leq & F(\mathbf{x}_{p_i}) -\sum_{k=p_i}^{q_i-1}\frac{a_k}{\vf} u_k - \sum_{k=p_i}^{q_i-1}\frac{a_k}{\vf} \frac{s^2}{64} + \sum_{k=p_i}^{q_i-1} \frac{a_k^2}{2}M
\end{eqnarray*}
where $u_k$ is defined in Equation \eqref{eq-def-u}. Thus,
\[
0< \sum_{k=p_i}^{q_i-1}a_k \leq \frac{64\vf}{s^2}\left(F(\mathbf{x}_{p_i}) - F(\mathbf{x}_{q_i}) -\frac{1}{\vf}\sum_{k=p_i}^{q_i-1} a_k u_k + \frac{M}{2}\sum_{k=p_i}^{q_i-1} a_k^2 \right)
\]
But $\lim_{k\rightarrow\infty} F(\mathbf{x}_k)$, $\sum_{k=0}^{\infty} a_k u_k$ and $\sum_{k=0}^{\infty} a_k^2$ all converge. So the above inequality implies that $\lim_{i\rightarrow\infty} \sum_{k=p_i}^{q_i-1}a_k = 0$, which contradicts Inequality \eqref{eq-sum-a-k-bound}. Hence, $\limsup_{k\rightarrow\infty} \|\nabla F(\mathbf{x}_k)\| =0$. 

Altogether, we have proved that $\lim_{k\rightarrow\infty} \|\nabla F(\mathbf{x}_k)\| =0$ and therefore $\lim_{k\rightarrow\infty} \nabla F(\mathbf{x}_k) =0$. This completes the proof of the lemma.
\end{proof}

\begin{proof}[Proof of Theorem \ref{thm-sgd}]
It is clear that Theorem \ref{thm-sgd} follows from Lemmas \ref{lemma-gradient-mixed-square-sum-converge} and \ref{lemma-nabla-F-converge}.
\end{proof}

\section{Augmented Back-Propagation Algorithm for Feed-Forward Networks}

In this section, we demonstrate that Theorem \ref{thm-sgd} can be applied to modify the back-propagation algorithm of feed-forward networks to force the algorithm to converge.

\subsection{Feed-forward networks} Recall that, as in Figure \ref{fig-feed-forward}, a feed-forward network of $l$ hidden layers is a function represented by a directed graph satisfying:
\begin{itemize}
  \item All nodes (vertices) are arranged in $l+2$ ordered layers. For minor notational convenience, we call the bottom layer the $0$-th layer, the layer above it the first layer,$\dots$, the top layer the $(l+1)$-th layer. Denote by $n_i$ the number of nodes in the $i$th layer.  
	\item The top layer ($(l+1)$-th layer) of nodes is call the input layer. Nodes in this layer, represented by circles in Figure \ref{fig-feed-forward}, take in input data. We denote by $x_j$ the input value for the $j$-th input node.
	\item The bottom layer ($0$-th layer) of nodes is called the output layer. Nodes in this layer, represented by triangles in Figure \ref{fig-feed-forward}, out put the final results.
	\item The layers of nodes between the input and output layers are called hidden layers. Each node in a hidden layer, represented by a box in Figure \ref{fig-feed-forward}, is a computation unit and is labeled by a differentiable function. For $i=1,\dots,l$ and $j=1,\dots,n_i$, we denote by $\s^i_j$ the function labeling the $j$-th node of the $i$-th layer. For notational simplicity, we use the convention that $\s^i_j(x)=x$ for $i=0$ or $l+1$ and all $x\in \R$.
	\item Arrows all point from a node in one layer to a node in the layer below. For any pair of nodes in adjacent layers, there is a single arrow pointing from the node in the upper layer to the one in the lower layer. Each arrow in this graph is labeled by a scalar, called the weight. We denote by $\lm^i_{j,j'}$ the weight of the arrow pointing from the $j$-th node in the $i$-th layer to the $j'$-th node in the $(i-1)$-th layer. Each arrow acts as a directed channel that multiplies the datum from its starting node by its weight and sends the product to its end node. 
	\item Each node in a hidden layer applies the function labeling it to the weighted sum of data from all of its incoming arrows, which generates its own output datum.
\end{itemize}
In other words, the feed-forward network in Figure \ref{fig-feed-forward} represents the function $F:\R^{n_{l+1}}\times \R^N\rightarrow \R^{n_0}$ given by $F([x_1,\dots,x_{n_{l+1}}]^T, \Lambda)=[z^0_1,\dots,z^0_{n_{0}}]^T$, where 
\begin{itemize}
	\item $N=\sum_{i=1}^{l+1} n_{i}n_{i-1}$,
	\item $\Lambda =[\lm^i_{j,j'}~|~i=1,\dots,l+1,~j=1,\dots,n_i,~j'=1,\dots, n_{i-1}] \in \R^N$,
	\item $Z=[z^i_j]$ is given by the following forward propagation
\begin{equation}\label{eq-function-forward-propagation}
\begin{cases}
z^{l+1}_j = x_j &\text{for } j=1,\dots,n_{l+1}, \\
z^i_j = \s^i_j (\sum_{j'=1}^{n_{i+1}} \lm^{i+1}_{j',j} z^{i+1}_{j'}) & \text{for } i=l,l-1,\dots,0 \text{ and } j =1,\dots,n_i. 
\end{cases}
\end{equation}
\end{itemize}

\begin{figure}[ht]
\[
\xymatrix{
x_1 \ar@{=}[d] && x_2 \ar@{=}[d] && \cdots \ar@{=}[d] && x_{n_{l+1}} \ar@{=}[d] \\
\bigcirc \ar^{\lambda^{l+1}_{1,1}}[d] \ar[drr] \ar[drrrr] \ar[drrrrrr] &&  \bigcirc \ar[dll] \ar[d] \ar[drr] \ar[drrrr] && \cdots \ar[dllll] \ar[dll] \ar[d] \ar[drr] &&  \bigcirc\ar[dllllll] \ar[dllll] \ar[dll] \ar^{\lambda^{l+1}_{n_{l+1},n_l}}[d] \\
\boxed{\sigma^{l}_{1}} \ar^{\lambda^{l}_{1,1}}[d] \ar[drr] \ar[drrrr] \ar[drrrrrr] &&  \boxed{\sigma^{l}_{2}} \ar[dll] \ar[d] \ar[drr] \ar[drrrr] && \cdots \ar[dllll] \ar[dll] \ar[d] \ar[drr] &&  \boxed{\sigma^{l}_{n_{l}}} \ar[dllllll] \ar[dllll] \ar[dll] \ar^{\lambda^{l}_{n_{l},n_{l-1}}}[d] \\
\cdots \ar[d] \ar[drr] \ar[drrrr] \ar[drrrrrr] &&  \cdots \ar[dll] \ar[d] \ar[drr] \ar[drrrr] && \cdots \ar[dllll] \ar[dll] \ar[d] \ar[drr] &&  \cdots \ar[dllllll] \ar[dllll] \ar[dll] \ar[d] \\
\boxed{\sigma^{1}_{1}} \ar^{\lambda^{1}_{1,1}}[d] \ar[drr] \ar[drrrr] \ar[drrrrrr] &&  \boxed{\sigma^{1}_{2}} \ar[dll] \ar[d] \ar[drr] \ar[drrrr] && \cdots \ar[dllll] \ar[dll] \ar[d] \ar[drr] &&  \boxed{\sigma^{1}_{n_{1}}} \ar[dllllll] \ar[dllll] \ar[dll] \ar^{\lambda^{1}_{n_{1},n_{0}}}[d] \\
\triangle \ar@{=}[d]  &&  \triangle \ar@{=}[d] && \cdots \ar@{=}[d] &&  \triangle \ar@{=}[d] \\
z^{0}_{1}   &&  z^{0}_{2}  && \cdots  &&  z^{0}_{n_{0}} 
} 
\]
\caption{A Feed-Forward Network}\label{fig-feed-forward} 
\end{figure}
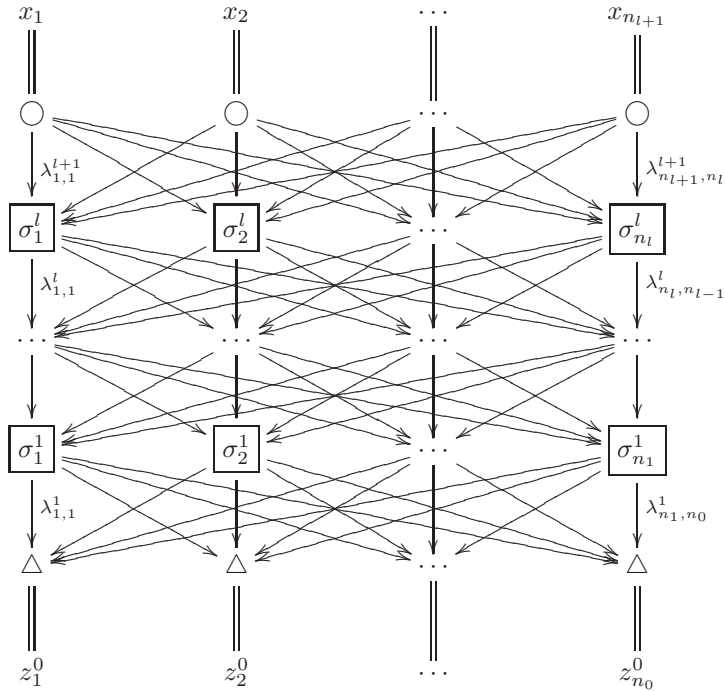

Let $\mathbf{x}=[x_1,\dots,x_{n_{l+1}}]^T\in \R^{n_{l+1}}$, $\Lambda =[\lm^i_{j,j'}~|~i=1,\dots,l+1,~j=1,\dots,n_i,~j'=1,\dots, n_{i-1}] \in \R^N$, $\mathbf{y}=[y_1,\dots,y_{n_{0}}]^T\in \R^{n_0}$ and $\mathbf{z}^0=[z^0_1,\dots,z^0_{n_{0}}]^T=F(\mathbf{x},\Lambda)\in \R^{n_0}$. For a fixed differentiable error function $E:\R^{n_0} \times \R^{n_0} \rightarrow \R$. The partial derivatives of $E(F(\mathbf{x},\Lambda),\mathbf{y})=E(\mathbf{z}^0,\mathbf{y})$ are given by the following backward propagation
\begin{equation}\label{eq-derivative-backward-propagation}
\begin{cases}
\frac{\partial}{\partial z^i_j} E(\mathbf{z}^0,\mathbf{y}) = \sum_{j'=1}^{n_{i-1}} \frac{\partial}{\partial z^{i-1}_{j'}} E(\mathbf{z}^0,\mathbf{y}) \cdot (\sigma^{i-1}_{j'})'(\sum_{p=1}^{n_i} \lambda^i_{p,j'}z^i_{p}) \cdot \lambda^i_{j,j'}, \\
\frac{\partial}{\partial \lm^i_{j,j'}} E(\mathbf{z}^0,\mathbf{y}) = \frac{\partial}{\partial z^{i-1}_{j'}} E(\mathbf{z}^0,\mathbf{y}) \cdot (\sigma^{i-1}_{j'})'(\sum_{p=1}^{n_i} \lambda^i_{p,j'}z^i_{p}) \cdot z^i_{j}.
\end{cases}
\end{equation}

Let $g:\R^{n_{l+1}}\rightarrow \R^{n_0}$ be a function. Assume that:
\begin{itemize}
	\item The sequence $\{(\mathbf{x}(k),\mathbf{y}(k))\}_{k=0}^\infty \subset \R^{n_{l+1}}\times \R^{n_0}$ satisfies $g(\mathbf{x}(k)) = \mathbf{y}(k)$.
	\item $\Lambda(0) =[\lm^i_{j,j'}(0)] \in \R^N$.
	\item $\{\eta_k\}_{k=0}^\infty \subset \R_+$.
\end{itemize}

The classical back-propagation algorithm in this set up is then given by the following steps:
\begin{itemize}
	\item Forward propagation:
	\begin{equation}\label{eq-forward-propagation}
\begin{cases}
z^{l+1}_j(k) = x_j(k), \\
z^i_j(k) = \s^i_j (\sum_{j'=1}^{n_{i+1}} \lm^{i+1}_{j',j}(k) z^{i+1}_{j'}(k)).
\end{cases}
\end{equation}
  \item Backward propagation:
	\begin{equation}\label{eq-backward-propagation}
\begin{cases}
\left(\frac{\partial}{\partial z^i_j} E(\mathbf{z}^0,\mathbf{y}(k))\right)\big{|}_{Z=Z(k),\Lambda=\Lambda(k)} \\
= \sum_{j'=1}^{n_{i-1}} \left(\frac{\partial}{\partial z^{i-1}_{j'}} E(\mathbf{z}^0,\mathbf{y}(k))\right)\big{|}_{Z=Z(k),\Lambda=\Lambda(k)}  \cdot (\sigma^{i-1}_{j'})'(\sum_{p=1}^{n_i} \lambda^i_{p,j'}(k)z^i_{p}(k)) \cdot \lambda^i_{j,j'}(k), \\
\\
\left(\frac{\partial}{\partial \lm^i_{j,j'}} E(\mathbf{z}^0,\mathbf{y}(k))\right)\big{|}_{Z=Z(k),\Lambda=\Lambda(k)} \\
= \left(\frac{\partial}{\partial z^{i-1}_{j'}} (E(\mathbf{z}^0,\mathbf{y}(k)))\right)\big{|}_{Z=Z(k),\Lambda=\Lambda(k)} \cdot (\sigma^{i-1}_{j'})'(\sum_{p=1}^{n_i} \lambda^i_{p,j'}(k)z^i_{p}(k)) \cdot z^i_{j}(k),
\end{cases}
\end{equation}
  \item Weight update:
	\begin{equation}\label{eq-weight-update}
	\lm^i_{j,j'}(k+1) = \lm^i_{j,j'}(k) - \eta_k \left(\frac{\partial}{\partial \lm^i_{j,j'}} E(\mathbf{z}^0,\mathbf{y}(k))\right)\big{|}_{Z=Z(k),\Lambda=\Lambda(k)}.
	\end{equation}
\end{itemize}
In the above, $\Lambda(k) = [\lm^i_{j,j'}(k)]$ and $Z(k) =[z^i_j(k)]$.

\subsection{Adequate augmentation.} The classical back-propagation algorithm has many interesting and important applications. But this algorithm is not convergent in general. We will use Theorem \ref{thm-sgd} to modify it and get a convergent algorithm. For this purpose, we introduce the concept of \textbf{adequate augmentation}. Before stating the definition, let us introduce some notations first. We view $\Lambda$ as a column vector in $\R^N$ and write
\begin{eqnarray}
\|\Lambda\|& = & \sqrt{\sum_{i=1}^{l+1}\sum_{j=1}^{n_i}\sum_{j'=1}^{n_{i-1}} (\lm^i_{j,j'})^2}, \\
\|\nabla_\Lambda P(\mathbf{x},\Lambda) \| & = & \sqrt{\sum_{i=1}^{l+1}\sum_{j=1}^{n_i}\sum_{j'=1}^{n_{i-1}} \left(\frac{\partial}{\partial \lm^i_{j,j'}}P(\mathbf{x},\Lambda)\right)^2}, \\
\Lambda^T \nabla_\Lambda P(\mathbf{x},\Lambda) & = & \sum_{i=1}^{l+1}\sum_{j=1}^{n_i}\sum_{j'=1}^{n_{i-1}} \lm^i_{j,j'}\frac{\partial}{\partial \lm^i_{j,j'}}P(\mathbf{x},\Lambda)
\end{eqnarray}
for any function $P:\R^{n_{l+1}}\times \R^N \rightarrow \R$.

\begin{definition}\label{def-adequate}
For a given choice of the error function $E(\mathbf{z},\mathbf{y})$, a function $\alpha:\R^N\rightarrow \R$ is called an \textbf{adequate augmentation} with respect to $E,F,g$ if
\begin{itemize}
	\item $\alpha$ is differentiable,
	\item $\nabla \alpha$ is locally Lipschitz, that is, for any $r>0$ there is a $c_r>0$ such that $\|\nabla \alpha (\Lambda) -\nabla \alpha (\Lambda')\| \leq c_r\|\Lambda - \Lambda'\|$ whenever $\|\Lambda\|,\|\Lambda'\|\leq r$,
	\item for every $\rho>0$, there is an $R_\rho>0$ such that $\Lambda^T \nabla_\Lambda \left(E(F(\mathbf{x},\Lambda),g(\mathbf{x}))+\alpha(\Lambda)\right)\geq 0$ whenever $\|\mathbf{x}\|\leq \rho$ and $\|\Lambda\| \geq R_\rho$.
\end{itemize}
\end{definition}

With some fairly nonrestrictive assumptions, it is not hard to find adequate augmentations for a given error function. We give some examples of adequate augmentations for the error function $E(\mathbf{z},\mathbf{y})=\|\mathbf{z}-\mathbf{y}\|^2$ below. The key to our constructions is the following lemma.

\begin{lemma}\label{lemma-error-bound}
Fix $E(\mathbf{z},\mathbf{y})=\|\mathbf{z}-\mathbf{y}\|^2$. Assume that 
\begin{itemize}
	\item the functions $\{\s^i_j\}$ are uniformly $C^1$-bounded, that is, there is an $M>0$ such that $\s^i_j(t)\leq M$ and $(\s^i_j)'(t)\leq M$ for all $i=1,\dots, l$, $j=1,\dots, n_i$ and $t\in\R$,
	\item the function $g(\mathbf{x})$ is locally bounded, that is, for every $\rho>0$, there is an $\Omega_\rho>0$ such that $\|g(\mathbf{x})\|\leq \Omega_\rho$ whenever $\|\mathbf{x}\|\leq \rho$.
\end{itemize}
Then, for every $\rho>0$, there is a $\Theta_\rho>0$ such that $\|\nabla_\Lambda E(F(\mathbf{x},\Lambda),g(\mathbf{x}))\| \leq \Theta_\rho (\|\Lambda\|^{l+1}+1)$ if $\|\mathbf{x}\|\leq \rho$.
\end{lemma}

\begin{proof}
For $i=1,\dots,l+1$, set $\Lambda^i =[\lm^i_{j,j'}~|~j=1,\dots,n_i,~j'=1,\dots, n_{i-1}] \in \R^{n_i n_{i-1}}$ and $\|\Lambda^i\|=\sqrt{\sum_{j=1}^{n_i}\sum_{j'=1}^{n_{i-1}} (\lm^i_{j,j'})^2}$. Fix a $\mathbf{y}\in\R^m$ satisfying $\|\mathbf{y}\|\leq \Omega_\rho$. We prove by induction that there are $\theta_\rho^1,\dots,\theta_\rho^{l+1}>0$ satisfying
\begin{equation}\label{eq-error-bound-induct}
\begin{cases}
\|\nabla_{\mathbf{z}^i} E(F(\mathbf{x},\Lambda),\mathbf{y})\| \leq \theta_\rho^i (\|\Lambda\|^{i+1}+1), \\
\|\nabla_{\Lambda^i} E(F(\mathbf{x},\Lambda),\mathbf{y})\| \leq \theta_\rho^i (\|\Lambda\|^{i}+1)
\end{cases}
\end{equation}
for $i=1,\dots,l+1$ and $\|\mathbf{x}\|\leq \rho$, where 
\begin{itemize}
	\item $\mathbf{z}^i=[z^i_1,\dots,z^i_{n_i}]^T$ and $\|\nabla_{\mathbf{z}^i} E(F(\mathbf{x},\Lambda),\mathbf{y}) \| = \sqrt{\sum_{j=1}^{n_i} \left(\frac{\partial}{\partial z^i_j}E(F(\mathbf{x},\Lambda),\mathbf{y})\right)^2}$,
	\item $\|\nabla_{\Lambda^i} E(F(\mathbf{x},\Lambda),\mathbf{y}) \| = \sqrt{\sum_{j=1}^{n_i}\sum_{j'=1}^{n_{i-1}} \left(\frac{\partial}{\partial \lm^i_{j,j'}}E(F(\mathbf{x},\Lambda),\mathbf{y})\right)^2}$.
\end{itemize}

For $i=1$, we have that 
\[
\frac{\partial}{\partial z^1_{j}} E(F(\mathbf{x},\Lambda),\mathbf{y}) = \sum_{j'=1}^{n_0} 2\lambda^1_{j,j'}(z^0_{j'}-y_{j'}) = \sum_{j'=1}^{n_0}2\lambda^1_{j,j'}((\sum_{p=1}^{n_1}\lm^1_{p,j'} z^1_p)-y_{j'})
\]
and
\[
\frac{\partial}{\partial \lambda^1_{j,j'}} E(F(\mathbf{x},\Lambda),\mathbf{y}) = \frac{\partial}{\partial \lambda^1_{j,j'}} E(\mathbf{z}^0,\mathbf{y}) = 2z^1_j(z^0_{j'}-y_{j'}) = 2z^1_j ((\sum_{p=1}^{n_1}\lm^1_{p,j'} z^1_p) -y_{j'}).
\]
So, for $\mathbf{x}\leq \rho$, we have that:
\begin{enumerate}[1.]
	\item \begin{eqnarray*}
|\frac{\partial}{\partial z^1_{j}} E(F(\mathbf{x},\Lambda),\mathbf{y})| & \leq &  \sum_{j'=1}^{n_0}2|\lambda^1_{j,j'}|((\sum_{p=1}^{n_1}|\lm^1_{p,j'}| M)+\Omega_\rho) \\
& \leq & \sum_{j'=1}^{n_0}2\|\Lambda^1\|((\sum_{p=1}^{n_1}\|\Lambda^1\| M)+\Omega_\rho) \leq 2n_1n_0 (M\|\Lambda^1\|^2 +\|\Lambda^1\|\Omega_\rho) \\
& \leq & 2n_1n_0 (M\|\Lambda^1\|^2 +(\|\Lambda^1\|^2+1)\Omega_\rho) = 2n_1n_0 ((M+\Omega_\rho)\|\Lambda^1\|^2 +\Omega_\rho).
\end{eqnarray*}
Therefore, 
\[
\|\nabla_{\mathbf{z}^i} E(F(\mathbf{x},\Lambda),\mathbf{y})\| \leq \sum_{j=1}^{n_1} |\frac{\partial}{\partial z^1_{j}} E(F(\mathbf{x},\Lambda),\mathbf{y})|  \leq 2n_1^2n_0 ((M+\Omega_\rho)\|\Lambda^1\|^2 +\Omega_\rho).
\]
	\item \[
|\frac{\partial}{\partial \lambda^1_{j,j'}} E(F(\mathbf{x},\Lambda),\mathbf{y})| \leq 2 M (\sum_{p=1}^{n_1}|\lm^1_{p,j'}|M+\Omega_\rho) \leq 2n_1 M^2 \|\Lambda^1\|+2M\Omega_\rho.
\]
Therefore, 
\[
\|\nabla_{\Lambda^1} E(F(\mathbf{x},\Lambda),\mathbf{y}) \| \leq \sum_{j=1}^{n_1}\sum_{j'=1}^{n_{0}} |\frac{\partial}{\partial \lm^i_{j,j'}}E(F(\mathbf{x},\Lambda),\mathbf{y})| \leq n_1 n_0 (2n_1 M^2 \|\Lambda^1\|+2M\Omega_\rho).
\]
\end{enumerate}
Thus, Inequality \eqref{eq-error-bound-induct} is true for $i=1$ and $\theta_\rho^1=\max\{2n_1^2n_0 ((M+\Omega_\rho), 2n_1^2 n_0  M^2, 2 n_1 n_0  M\Omega_\rho\}$.

Now assume that Inequality \eqref{eq-error-bound-induct} is true for $i-1\geq 1$. For $\|\mathbf{x}\|\leq \rho$, by the backward propagation \eqref{eq-derivative-backward-propagation}, we have that: 
\begin{enumerate}[1.]
	\item \begin{eqnarray*}
& & |\frac{\partial}{\partial z^i_j} E(\mathbf{z}^0,\mathbf{y})| = |\sum_{j'=1}^{n_{i-1}} \frac{\partial}{\partial z^{i-1}_{j'}} E(\mathbf{z}^0,\mathbf{y}) \cdot (\sigma^{i-1}_{j'})'(\sum_{p=1}^{n_i} \lambda^i_{p,j'}z^i_{p}) \cdot \lambda^i_{j,j'}| \\
& \leq &  \sum_{j'=1}^{n_{i-1}} |\frac{\partial}{\partial z^{i-1}_{j'}} E(\mathbf{z}^0,\mathbf{y})| \cdot |(\sigma^{i-1}_{j'})'(\sum_{p=1}^{n_i} \lambda^i_{p,j'}z^i_{p})| \cdot |\lambda^i_{j,j'}| \\
& \leq & \sum_{j'=1}^{n_{i-1}} \theta^{i-1}_\rho (\|\Lambda\|^i+1) \cdot M \cdot \|\Lambda\| = n_{i-1}  M  \theta^{i-1}_\rho (\|\Lambda\|^{i+1}+\|\Lambda\|) \leq n_{i-1}  M  \theta^{i-1}_\rho (2\|\Lambda\|^{i+1}+1).
\end{eqnarray*}
Therefore, 
\[
\|\nabla_{\mathbf{z}^i} E(F(\mathbf{x},\Lambda),\mathbf{y})\| \leq \sum_{j=1}^{n_i} |\frac{\partial}{\partial z^i_j} E(\mathbf{z}^0,\mathbf{y})| \leq n_in_{i-1}  M  \theta^{i-1}_\rho (2\|\Lambda\|^{i+1}+1).
\]
  \item \begin{eqnarray*}
	&& |\frac{\partial}{\partial \lm^i_{j,j'}} E(\mathbf{z}^0,\mathbf{y})| = |\frac{\partial}{\partial z^{i-1}_{j'}} E(\mathbf{z}^0,\mathbf{y}) \cdot (\sigma^{i-1}_{j'})'(\sum_{p=1}^{n_i} \lambda^i_{p,j'}z^i_{p}) \cdot z^i_{j}| \\
	& \leq & \|\nabla_{\mathbf{z}^{i-1}} E(F(\mathbf{x},\Lambda),\mathbf{y})\| \cdot M \cdot M \leq M^2 \theta^{i-1}_\rho (\|\Lambda\|^i +1).
\end{eqnarray*}
Consequently,
\[
\|\nabla_{\Lambda^i} E(F(\mathbf{x},\Lambda), \mathbf{y})\| \leq \sum_{j=1}^{n_i}\sum_{j'=1}^{n_{i-1}}|\frac{\partial}{\partial \lm^i_{j,j'}} E(\mathbf{z}^0,\mathbf{y})| \leq n_i n_{i-1}M^2 \theta^{i-1}_\rho (\|\Lambda\|^i +1).
\]
\end{enumerate}
Thus, Inequality \eqref{eq-error-bound-induct} is true for $i$ and $\theta_\rho^i=\max\{2n_in_{i-1}  M  \theta^{i-1}_\rho, n_i n_{i-1}M^2 \theta^{i-1}_\rho\}$. This proves Inequality \eqref{eq-error-bound-induct} for $i=1,\dots, l+1$.

Finally, we have that $\|g(\mathbf{x})\|\leq \Omega_\rho$ for any $\|\mathbf{x}\|\leq \rho$ and, consequently, 
\[
\|\nabla_\Lambda E(F(\mathbf{x},\Lambda),g(\mathbf{x}))\| \leq \sum_{i=1}^{l+1} \|\nabla_{\Lambda^i} E(F(\mathbf{x},\Lambda),g(\mathbf{x}))\| \leq \sum_{i=1}^{l+1} \theta_\rho^i (\|\Lambda\|^{i}+1) \leq \sum_{i=1}^{l+1} \theta_\rho^i (\|\Lambda\|^{l+1}+2).
\]
So the lemma is true for $\Theta_\rho = 2\sum_{i=1}^{l+1} \theta_\rho^i$.
\end{proof}

Based on Lemma \ref{lemma-error-bound}, we have the following examples of adequate augmentations for the error function $E(\mathbf{z},\mathbf{y})=\|\mathbf{z}-\mathbf{y}\|^2$.

\begin{corollary}\label{cor-sample-aa}
Define 
\begin{enumerate}
	\item $\alpha_1(\Lambda) = \delta \|\Lambda\|^t$, where $\delta, t \in \R$ satisfy $\delta>0$ and $t>l+2$,
	\item $\alpha_2(\Lambda) = \begin{cases} 0 &\text{if } \|\Lambda\|< r, \\ \delta (\|\Lambda\|-r)^t & \text{if }\|\Lambda\|\geq r,\end{cases}$ where $\delta, r,t \in \R$ satisfy $\delta, r>0$ and $t>l+2$,
	\item $\alpha_3 (\Lambda) = \begin{cases} 0 &\text{if } \|\Lambda\|< r, \\ e^{\|\Lambda\|-r}-\sum_{p=0}^q \frac{(\|\Lambda\|-r)^p}{p!} & \text{if }\|\Lambda\|\geq r,\end{cases}$ where $r>0$ and $q\geq 1$.
\end{enumerate}
Then, $\alpha_1,\alpha_2,\alpha_3$ are all adequate augmentations for the error function $E(\mathbf{z},\mathbf{y})=\|\mathbf{z}-\mathbf{y}\|^2$ under the assumptions of Lemma \ref{lemma-error-bound}.
\end{corollary}

\begin{proof}
One can see that $\alpha_i$ is differentiable and that $\nabla \alpha_i$ is locally Lipschitz for $i=1,2,3$. A straightforward computation shows that, for $i=1,2,3$, $\Lambda^T \nabla \alpha_i (\Lambda) > \Theta_\rho(\|\Lambda\|^{l+2}+1)$ if $\|\Lambda\|\gg 1$. So, by Lemma \ref{lemma-error-bound}, we have that, for $i=1,2,3$,
\[
\Lambda^T \nabla_\Lambda \left(E(F(\mathbf{x},\Lambda),g(\mathbf{x}))+\alpha_i(\Lambda)\right)\geq \Lambda^T \nabla \alpha_i (\Lambda) - \|\Lambda\|\cdot\|\nabla_\Lambda \left(E(F(\mathbf{x},\Lambda),g(\mathbf{x}))\right)\| >0 \text{ for } \|\Lambda\|\gg 1.
\]
Thus, $\alpha_1,\alpha_2,\alpha_3$ are all adequate augmentations for the error function $E(\mathbf{z},\mathbf{y})=\|\mathbf{z}-\mathbf{y}\|^2$.
\end{proof}

\subsection{Augmented back-propagation algorithm} Next, we use Theorem \ref{thm-sgd} to establish a convergent augmented back-propagation algorithm.

\begin{proposition}\label{prop-AA-BP}
Let $F:\R^{n_{l+1}}\times \R^N\rightarrow \R^{n_0}$ be the function represented by the feed-forward network in Figure \ref{fig-feed-forward}. Fix a $0 <\rho\leq \infty$. Assume that:
\begin{enumerate}
	\item The functions $\{\s^i_j\}$ are uniformly $C^2$-bounded, that is, there is a $M>0$ such that $\s^i_j(t)\leq M$, $(\s^i_j)'(t)\leq M$ and $(\s^i_j)''(t)\leq M$ for all $i=1,\dots, l$, $j=1,\dots, n_i$ and $t\in\R$.
	\item $\mu$ is a probability measure on $B^{n_{l+1}}_\rho(\mathbf{0})$. $\{\mathbf{x}(k)\}_{k=0}^\infty$ is a sequence of independent random variables with identical probability distribution taking values in $B^{n_{l+1}}_\rho(\mathbf{0})$, where the probability distribution of each $\mathbf{x}(k)$ is $\mu$.
	\item $g:B^{n_{l+1}}_\rho(\mathbf{0})\rightarrow \R^{n_0}$ is a bounded function, that is, there is a $G>0$ such that $\|g(\mathbf{x})\|\leq G$ for all $\mathbf{x} \in B^{n_{l+1}}_\rho(\mathbf{0})$. And $\mathbf{y}(k)=g(\mathbf{x}(k))$ for $k=0,1,2,\dots$
	\item $E:\R^{n_0}\times\R^{n_0}\rightarrow \R$ is an error function that is $C^2$-bounded with respect to the first $n_0$ variables in $B^{n_{0}}_r \times B^{n_{0}}_G$ for every $r>0$. That is, there is a $\Gamma_r>0$ satisfying 
	\[
	|E(\mathbf{z},\mathbf{y})| + \|\nabla_{\mathbf{z}}E(\mathbf{z},\mathbf{y})\| + \sqrt{\Tr ((H_{\mathbf{z}}E(\mathbf{z},\mathbf{y}))^T H_{\mathbf{z}}E(\mathbf{z},\mathbf{y}))} \leq \Gamma_r\]
	for all $\mathbf{z}\in B^{n_{0}}_r$ and $\mathbf{y}\in B^{n_{0}}_G$, where $H_{\mathbf{z}}E(\mathbf{z},\mathbf{y})$ is the Hessian matrix of $E(\mathbf{z},\mathbf{y})$ with respect to $\mathbf{z}$.
	\item $\alpha:\R^N\rightarrow \R$ is an adequate augmentation with respect to $E,F,g$. Fix an $R_0>0$ such that 
	\[
	\Lambda^T \nabla_\Lambda \left(E(F(\mathbf{x},\Lambda),g(\mathbf{x}))+\alpha(\Lambda)\right)\geq 0
	\]
	whenever $\|\mathbf{x}\|\leq \rho$ and $\|\Lambda\| \geq R_0$.
	\item $\{\eta_k\}_{k=0}^\infty$ is a sequence of positive numbers satisfying $\sum_{k=0}^\infty\eta_k =\infty$ and $\sum_{k=0}^\infty \eta_k^2 <\infty$. Set $H = \sup\{\eta_k~|~k\geq 0\}$.
	\item 
	\begin{itemize}
		\item $\Lambda(0)=[\lambda^{i}_{j,j'}(0)]$ is any fixed element of $\R^N$, 
		\item $R_1 = \max\{\sqrt{\|\Lambda(0)\|^2 + \sum_{k=0}^\infty \eta_k^2},\sqrt{R_0^2 +2HR_0 +\sum_{k=0}^\infty \eta_k^2}\}$,
		\item $\Phi =\sup \{\|\nabla_\Lambda \left(E(F(\mathbf{x},\Lambda),g(\mathbf{x}))+\alpha(\Lambda)\right)\|~|~ (\mathbf{x},\Lambda) \in  B^{n_{l+1}}_\rho(\mathbf{0}) \times B^N_{R_1}(\mathbf{0})\}<\infty$,
		\item Fix a positive number $\vf$ satisfying $\vf\geq \Phi$.
	\end{itemize}
\end{enumerate}
Define a sequence of random variables $\{\Lambda(k) =[\lambda^{i}_{j,j'}(k)]\}_{k=0}^\infty$ by the following:
\begin{itemize}
	\item Forward propagation given inductively by \eqref{eq-forward-propagation};
  \item Backward propagation given inductively by \eqref{eq-backward-propagation};
  \item \textbf{Augmented weight update} given by
	\begin{equation}\label{eq-weight-update-aug}
	\lm^i_{j,j'}(k+1) = \lm^i_{j,j'}(k) - \frac{\eta_k}{\vf} \left(\frac{\partial}{\partial \lm^i_{j,j'}} (E(\mathbf{z}^0,\mathbf{y}(k))+\alpha(\Lambda))\right)\big{|}_{Z=Z(k),\Lambda=\Lambda(k)}.
	\end{equation}
\end{itemize}
Define the mean error function $\mathcal{E}: \R^N\rightarrow \R$ by $\mathcal{E}(\Lambda) = \int_{B^{n_{l+1}}_\rho(\mathbf{0})}E(F(\ast,\Lambda),g(\ast)) d\mu$. Then 
\begin{itemize}
	\item $\{\mathcal{E}(\Lambda(k))+\alpha(\Lambda(k))\}$ converges almost surely to a finite number,
	\item $\{\nabla\mathcal{E}(\Lambda(k))+\nabla\alpha(\Lambda(k))\}$ converges almost surely to $\mathbf{0}$,
	\item any limit point of $\{\Lambda(k)\}$ is almost surely a stationary point of the function $\mathcal{E} + \alpha$.
\end{itemize}
\end{proposition}

\begin{proof}
It is straightforward to check that Proposition \ref{prop-AA-BP} follows from the application of Theorem \ref{thm-sgd} to 
\begin{itemize}
	\item the function $f:\R^N\times B^{n_{l+1}}_\rho\rightarrow \R$ given by $f(\Lambda,\mathbf{x}) := E(F(\mathbf{x},\Lambda),g(\mathbf{x})) + \alpha(\Lambda)$,
	\item the inputs $\{\mathbf{x}(k)\}_{k=0}^\infty$,
	\item the sequence $\{\eta_k\}_{k=0}^\infty$.
\end{itemize}
\end{proof}

\section{Augmented Back-Propagation Algorithm for Acyclic Neural Networks}

The augmented back-propagation algorithm is not limited to feed-forward networks. We give in this section the augmented back-propagation algorithm for neural networks whose underlying directed graphs are acyclic. Such acyclic neural networks are straightforward generalizations of feed-forward networks and include, for example, feed-forward networks with biases. 

\subsection{Acyclic Neural Networks}

Let $G$ be a finite directed graph without loops and parallel edges. That is, 
\begin{itemize}
	\item the set $\mathfrak{V}$ of vertices of $G$ is finite,
	\item the set $\mathfrak{E}$ of edges of $G$ is a subset of $(\mathfrak{V}\times \mathfrak{V}) \setminus \{(v,v)~|~v\in \mathfrak{V}\}$.
\end{itemize}
Every edge $e\in \mathfrak{E}$ has a starting vertex $s(e)$ and a terminal vertex $t(e)$. For any $v\in \mathfrak{V}$, the set of edges pointing out of (resp. into) $v$ is denoted by $Out(v)=\{e\in \mathfrak{E}~|~ s(e)=v\}$ (resp. $In(v)=\{e\in \mathfrak{E} ~|~ t(e)=v\}$.) We also set $\mathfrak{V}_{in}=\{v\in \mathfrak{V}~|~ In(v)=\emptyset\}$ and $\mathfrak{V}_{out}=\{v\in \mathfrak{V} ~|~Out(v)=\emptyset\}$. $\mathfrak{V}_{in}$ and $\mathfrak{V}_{out}$ are called the sets of input and output vertices of $G$, respectively. A directed path in $G$ is a sequence $\{e_i\}_{i=1}^l$ of distinct edges satisfying 
\begin{itemize}
	\item $t(e_i)=s(e_{i+1})$ for $i=1,\dots,l-1$,
	\item $t(e_1),\dots,t(e_l)$ are distinct.
\end{itemize}
$\{e_i\}_{i=1}^l$ is called a directed cycle if we further have that $t(e_l)=s(e_1)$.

In the rest of this section, we assume that $G$ is \textbf{acyclic}, that is, $G$ does not contain any directed cycles. For each vertex $v$, we define the \textbf{depth} of $v$ to be
\begin{equation}\label{eq-def-d}
d(v):=\begin{cases}
0 & \text{if }v\in \mathfrak{V}_{in}, \\
\max\{l~|~ \text{there is a directed path } \{e_i\}_{i=1}^l \text{ in } G \text{ satisfying } t(e_l)=v\} & \text{if } v\notin \mathfrak{V}_{in}.
\end{cases}
\end{equation}
We also define the \textbf{height} of $v$ to be 
\begin{equation}\label{eq-def-h}
h(v):=\begin{cases}
0 & \text{if }v\in \mathfrak{V}_{out}, \\
\max\{l~|~ \text{there is a directed path } \{e_i\}_{i=1}^l \text{ in } G \text{ satisfying } s(e_1)=v\} & \text{if } v\notin \mathfrak{V}_{out}.
\end{cases}
\end{equation}
$d(v)$ and $h(v)$ are well defined since $G$ is acyclic. We have the following simple observations.

\begin{lemma}\label{lemma-d-h-consecutive}
\begin{enumerate}
  \item $\max\{h(v)~|~ v\in \mathfrak{V}\}=\max\{d(v)~|~v\in \mathfrak{V}\}$. We call this common value the \textbf{height} $H(G)$ of $G$ from now on.
	\item $\{d(v)~|~ v\in \mathfrak{V}\}=\{0,1,\dots,H(G)\}$ and $\{h(v)~|~ v\in v\}=\{0,1,\dots,H(G)\}$.
	\item $d(s(e))\leq d(t(e))-1$ and $h(t(e))\leq h(s(e))-1$ for any edge $e \in \mathfrak{E}$. 
\end{enumerate}
\end{lemma}

A neural network underlay by $G$ consists of 
\begin{itemize}
	\item a weight function $\lambda:\mathfrak{E}\rightarrow \R$,
	\item an assignment of a function $\sigma_v:\R\rightarrow \R$ to each vertex $v\in \mathfrak{V}\setminus (\mathfrak{V}_{in}\cup \mathfrak{V}_{out})$.
\end{itemize}
Order the input and output vertices as $\mathfrak{V}_{in}=\{v_1,\dots, v_n\}$ and $\mathfrak{V}_{out}=\{u_1,\dots,u_m\}$. And denote by $N$ the number of edges in $G$. Note that, with an ordering of the edges of $G$, $\lambda$ can be viewed as a vector in $\R^N$. Then the above neural network defines a function $F:\R^n \times \R^N \rightarrow \R^m$ by $F(\mathbf{x},\lambda)=\mathbf{z}$ where $\mathbf{x}=[x_1,\dots,x_n]^T\in\R^n$, $\lambda:\mathfrak{E}\rightarrow \R$ and $\mathbf{z}=[z_1,\dots,z_m]^T\in\R^m$ are related by the following forward propagation
\begin{equation}\label{eq-forward-G}
\begin{cases}
z(v_i)=x_i & \text{for } i=1,2,\dots,n \\
z(v)=\sigma_v(\sum_{t(e)=v} \lambda(e)z(s(e))) & \text{for } v\in \mathfrak{V}\setminus (\mathfrak{V}_{in}\cup \mathfrak{V}_{out}), \\
z_j=z(u_j)= \sum_{t(e)=u_j} \lambda(e)z(s(e)) & \text{for } j=1,2,\dots,m.
\end{cases}
\end{equation}
It follows from Lemma \ref{lemma-d-h-consecutive} that forward propagation \eqref{eq-forward-G} is a well defined inductive definition. From now on, we assume that $\sigma_v$ is differentiable for each $v\in \mathfrak{V}$. It is straightforward to check that $F$ is differentiable in both $\mathbf{x}$ and $\lambda$.

Fix a differentiable error function $E:\R^m\times\R^m\rightarrow \R$. The function $E(F(\mathbf{x},\lambda),\mathbf{y})=E(\mathbf{z},\mathbf{y})$ is differentiable. Its partial derivatives with respect to $z(v)$ and $\lambda(e)$ can be computed by the following backward propagation. 
\begin{equation}\label{eq-backward-G}
\begin{cases}
\frac{\partial}{\partial z(v)}E(F(\mathbf{x},\lambda),\mathbf{y})= \sum_{e\in Out(v)} \left(\frac{\partial}{\partial z(t(e))}E(F(\mathbf{x},\lambda),\mathbf{y})\right) \cdot \lambda(e)\cdot(\sigma_{t(e)})'(\sum_{\hat{e}\in In(t(e))} \lambda(\hat{e})z(s(\hat{e}))),\\
\frac{\partial}{\partial \lambda(e)}E(F(\mathbf{x},\lambda),\mathbf{y}) = \left(\frac{\partial}{\partial z(t(e))}E(F(\mathbf{x},\lambda),\mathbf{y})\right)\cdot z(s(e)) \cdot(\sigma_{t(e)})'(\sum_{\hat{e}\in In(t(e))} \lambda(\hat{e})z(s(\hat{e}))).
\end{cases}
\end{equation}
By Lemma \ref{lemma-d-h-consecutive} again, backward propagation \eqref{eq-backward-G} is a well defined inductive computation.

Let $g:\R^{n}\rightarrow \R^{n}$ be a function. Assume that:
\begin{itemize}
	\item The sequence $\{(\mathbf{x}(k),\mathbf{y}(k))\}_{k=0}^\infty \subset \R^{n}\times \R^{m}$ satisfies $g(\mathbf{x}(k)) = \mathbf{y}(k)$.
	\item Fix a $\lambda_0:\mathfrak{E}\rightarrow \R$.
	\item $\{\eta_k\}_{k=0}^\infty \subset \R_+$.
\end{itemize} 

For the above setup, the classical back-propagation algorithm is given inductively by the following steps:
\begin{itemize}
	\item Forward propagation:
\begin{equation}\label{eq-forward-bp-G}
\begin{cases}
z(v_i,k)=x_i(k) & \text{for } i=1,2,\dots,n \\
z(v,k)=\sigma_v(\sum_{t(e)=v} \lambda_k(e)z(s(e),k)) & \text{for } v\in \mathfrak{V}\setminus (\mathfrak{V}_{in}\cup \mathfrak{V}_{out}), \\
z_j(k)=z(u_j,k)= \sum_{t(e)=u_j} \lambda(e)z(s(e),k) & \text{for } j=1,2,\dots,m,
\end{cases}
\end{equation}
where $\mathbf{x}(k)=[x_1(k),\dots,x_n(k)]^T\in \R^n$.
  \item Backward propagation:
\begin{equation}\label{eq-backward-bp-G}
\begin{cases}
\left(\frac{\partial}{\partial z(v)}E(F(\mathbf{x},\lambda),\mathbf{y})\right)|_{z(v)=z(v,k)~\forall v\in \mathfrak{V}, ~\lambda(e)=\lambda_k(e)~\forall e\in \mathfrak{E}} \\
= \sum_{e\in Out(v)} \lambda_k(e)\cdot(\sigma_{t(e)})'(\sum_{\hat{e}\in In(t(e))} \lambda_k(\hat{e})z(s(\hat{e}),k)) \cdot \left(\frac{\partial}{\partial z(t(e))}E(F(\mathbf{x},\lambda),\mathbf{y})\right)|_{z(v)=z(v,k)~\forall v\in \mathfrak{V}, ~\lambda(e)=\lambda_k(e)~\forall e\in \mathfrak{E}} ,\\
\\
\left(\frac{\partial}{\partial \lambda(e)}E(F(\mathbf{x},\lambda),\mathbf{y})\right)|_{z(v)=z(v,k)~\forall v\in \mathfrak{V}, ~\lambda(e)=\lambda_k(e)~\forall e\in \mathfrak{E}} \\ 
= z(s(e),k) \cdot(\sigma_{t(e)})'(\sum_{\hat{e}\in In(t(e))} \lambda_k(\hat{e})z(s(\hat{e}),k)) \cdot \left(\frac{\partial}{\partial z(t(e))}E(F(\mathbf{x},\lambda),\mathbf{y})\right)|_{z(v)=z(v,k)~\forall v\in \mathfrak{V}, ~\lambda(e)=\lambda_k(e)~\forall e\in \mathfrak{E}}.
\end{cases}
\end{equation}
  \item Weight update:
	\begin{equation}\label{eq-weight-update-G-bp}
	\lm_{k+1}(e) = \lm_{k}(e) - \eta_k \left(\frac{\partial}{\partial \lambda(e)}E(F(\mathbf{x},\lambda),\mathbf{y})\right)|_{z(v)=z(v,k)~\forall v\in \mathfrak{V}, ~\lambda(e)=\lambda_k(e)~\forall e\in \mathfrak{E}}.
	\end{equation}
\end{itemize}

\subsection{Adequate augmentations} The definition and properties of adequate augmentations generalize easily to acyclic neural networks. Let us introduce the following notations first. We view a function $\lambda:\mathfrak{E}\rightarrow \R$ as a column vector in $\R^N$ and write
\begin{eqnarray}
\|\lambda\|& = & \sqrt{\sum_{e\in \mathfrak{E}} (\lm(e))^2}, \\
\|\nabla_\lambda P(\mathbf{x},\lambda) \| & = & \sqrt{\sum_{e\in \mathfrak{E}} \left(\frac{\partial}{\partial \lm(e)}P(\mathbf{x},\lambda)\right)^2}, \\
\lambda^T \nabla_\lambda P(\mathbf{x},\lambda) & = & \sum_{e\in \mathfrak{E}} \lm(e)\frac{\partial}{\partial \lm(e)}P(\mathbf{x},\lambda)
\end{eqnarray}
for any function $P:\R^{n}\times \R^N \rightarrow \R$.

\begin{definition}\label{def-adequate-G}
For a given choice of the error function $E(\mathbf{z},\mathbf{y})$, a function $\alpha:\R^N\rightarrow \R$ is called an \textbf{adequate augmentation} with respect to $E,F,g$ if
\begin{itemize}
	\item $\alpha$ is differentiable,
	\item $\nabla \alpha$ is locally Lipschitz, that is, for any $r>0$ there is a $c_r>0$ such that $\|\nabla \alpha (\lambda) -\nabla \alpha (\lambda')\| \leq c_r\|\lambda - \lambda'\|$ whenever $\|\lambda\|,\|\lambda'\|\leq r$,
	\item for every $\rho>0$, there is an $R_\rho>0$ such that $\lambda^T \nabla_\lambda \left(E(F(\mathbf{x},\lambda),g(\mathbf{x}))+\alpha(\lambda)\right)\geq 0$ whenever $\|\mathbf{x}\|\leq \rho$ and $\|\lambda\| \geq R_\rho$.
\end{itemize}
\end{definition}

As in the case of feed-forward networks, it is not hard to find adequate augmentations for a given error functions satisfying some non-restrictive properties. Again, we give some examples of adequate augmentations for the error function $E(\mathbf{z},\mathbf{y})=\|\mathbf{z}-\mathbf{y}\|^2$ below. The key to our constructions is Lemma \ref{lemma-error-bound-G}, which is a straightforward generalizations of Lemma \ref{lemma-error-bound}.

\begin{lemma}\label{lemma-error-bound-G}
Fix $E(\mathbf{z},\mathbf{y})=\|\mathbf{z}-\mathbf{y}\|^2$. Assume that 
\begin{itemize}
	\item the functions $\{\s_v~|~v\in \mathfrak{V}\}$ are uniformly $C^1$-bounded, that is, there is an $M>0$ such that $\s_v(t)\leq M$ and $(\s_v)'(t)\leq M$ for all $v\in \mathfrak{V}$,
	\item the function $g(\mathbf{x})$ is locally bounded, that is, for every $\rho>0$, there is an $\Omega_\rho>0$ such that $\|g(\mathbf{x})\|\leq \Omega_\rho$ whenever $\|\mathbf{x}\|\leq \rho$.
\end{itemize}
Then, for every $\rho>0$, there is a $\Theta_\rho>0$ such that $\|\nabla_\lambda E(F(\mathbf{x},\lambda),g(\mathbf{x}))\| \leq \Theta_\rho (\|\lambda\|^{H(G)}+1)$ if $\|\mathbf{x}\|\leq \rho$.
\end{lemma}

\begin{proof}
The proof of Lemma \ref{lemma-error-bound-G} follows more or less that of Lemma \ref{lemma-error-bound}. Without loss of generality, we assume that $M\geq \max\{1,\rho\}$. Fix a $\mathbf{y}=[y_1,\dots,y_m]^T\in\R^m$ satisfying $\|\mathbf{y}\|\leq \Omega_\rho$. We first prove that there is a function $\theta:\mathfrak{V} \rightarrow \R_{>0}$ satisfying
\begin{equation}\label{eq-lemma-error-bound-G-ind}
\begin{cases}
\left| \frac{\partial}{\partial z(v)} E(F(\mathbf{x},\lambda),\mathbf{y})\right| \leq \theta(v)(\|\lambda\|^{h(v)+1}+1) & \text{for every } v\in \mathfrak{V}, \\
\left| \frac{\partial}{\partial \lambda(e)} E(F(\mathbf{x},\lambda),\mathbf{y})\right| \leq M^2\theta(t(e))(\|\lambda\|^{h(t(e))+1}+1) & \text{for every } e\in \mathfrak{E}.
\end{cases}
\end{equation} 
By Lemma \ref{lemma-d-h-consecutive}, the function $\theta$ can be constructed by an induction on $h(v)$.

Assume that $h(v)=0$, then $v=u_j\in \mathfrak{V}_{out}$ for some $j$ and 
\[
\left|\frac{\partial}{\partial z(v)} E(F(\mathbf{x},\lambda),\mathbf{y})\right| = 2|z(u_j)-y_j| \leq 2|z(u_j)| + 2|y_j| = 2\left|\sum_{e\in In(u_j)}\lambda(e)z(s(e))\right|+ 2|y_j| \leq 2\|\lambda\|M\sqrt{\# In(u_j)} + 2\Omega_\rho,
\]
where $\# In(u_j)$ is the cardinality of $In(u_j)$. So we can define that $\theta(v)=\theta(u_j)=\max\{2M\sqrt{\# In(u_j)}, 2\Omega_\rho\}$. 

Now assume that $\theta$ is defined for all $u\in \mathfrak{V}$ satisfying $0\leq h(u)\leq h-1$ and that $v \in \mathfrak{V}$ satisfies $h(v)=h$. By Lemma \ref{lemma-d-h-consecutive}, we have that $h(t(e))\leq h-1$ for all $e \in Out(v)$. So, by Equation \eqref{eq-backward-G}, we have
\begin{eqnarray*}
\left|\frac{\partial}{\partial z(v)}E(F(\mathbf{x},\lambda),\mathbf{y})\right| & \leq & \sum_{e\in Out(v)} \left|\frac{\partial}{\partial z(t(e))}E(F(\mathbf{x},\lambda),\mathbf{y})\right| \cdot |\lambda(e)| \cdot \left|(\sigma_{t(e)})'(\sum_{\hat{e}\in In(t(e))} \lambda(\hat{e})z(s(\hat{e})))\right| \\
& \leq & \sum_{e\in Out(v)} \theta(t(e))(\|\lambda\|^{h(t(e))+1}+1) \cdot \|\lambda\| \cdot M \\
& = & M \sum_{e\in Out(v)} \theta(t(e))(\|\lambda\|^{h(t(e))+2}+\|\lambda\|) \leq M \sum_{e\in Out(v)} \theta(t(e))(2\|\lambda\|^{h+1}+2).
\end{eqnarray*}
Here, we use the convention that $\sigma_u(z)=z$ for all $u\in \R$ if $u\in \mathfrak{V}_{out}$. So we can define that $\theta(v)=2M \sum_{e\in Out(v)} \theta(t(e))$. This defines $\theta (v)$ for any $v$ with $h(v)=h$, which completes the inductive construction of $\theta$ and proves that the first inequality in \eqref{eq-lemma-error-bound-G-ind} is satisfied.

Using Equation \eqref{eq-backward-G} again, we get that, for any $e\in \mathfrak{E}$,
\[
\left|\frac{\partial}{\partial \lambda(e)}E(F(\mathbf{x},\lambda),\mathbf{y})\right| = \left|\frac{\partial}{\partial z(t(e))}E(F(\mathbf{x},\lambda),\mathbf{y})\right|\cdot |z(s(e))| \cdot \left|(\sigma_{t(e)})'(\sum_{\hat{e}\in In(t(e))} \lambda(\hat{e})z(s(\hat{e})))\right| 
\]
Note that $|z(s(e))| \leq  \max\{M,\rho\}=M$ and $\left|(\sigma_{t(e)})'(\sum_{\hat{e}\in In(t(e))} \lambda(\hat{e})z(s(\hat{e})))\right| \leq \max\{M,1\}=M$. So, the second inequality in \eqref{eq-lemma-error-bound-G-ind} follows from the first one.

Finally, for any $\mathbf{x} \in \R^n$ satisfying $\|\mathbf{x}\|\leq \rho$ and $\lambda \in \R^N$, by the second inequality in \eqref{eq-lemma-error-bound-G-ind}, we have
\begin{eqnarray*}
&& \|\nabla_\lambda E(F(\mathbf{x},\lambda),g(\mathbf{x}))\|  \leq  \sum_{e\in \mathfrak{E}}\left| \frac{\partial}{\partial \lambda(e)} E(F(\mathbf{x},\lambda),g(\mathbf{x}))\right| \\
& \leq & M^2 \sum_{e\in \mathfrak{E}}\theta(t(e))(\|\lambda\|^{h(t(e))+1}+1) \leq  M^2 \sum_{e\in \mathfrak{E}}\theta(t(e))(\|\lambda\|^{h(G)}+2).
\end{eqnarray*}
So $\Theta_\rho = 2M^2 \sum_{e\in \mathfrak{E}}\theta(t(e))$ satisfies the requirements in the lemma.
\end{proof}

\begin{corollary}\label{cor-sample-aa-G}
Define 
\begin{enumerate}
	\item $\alpha_1(\lambda) = \delta \|\lambda\|^t$, where $\delta, t \in \R$ satisfy $\delta>0$ and $t>h(G)+1$,
	\item $\alpha_2(\lambda) = \begin{cases} 0 &\text{if } \|\lambda\|< r, \\ \delta (\|\lambda\|-r)^t & \text{if }\|\lambda\|\geq r,\end{cases}$ where $\delta, r,t \in \R$ satisfy $\delta, r>0$ and $t>h(G)+1$,
	\item $\alpha_3 (\lambda) = \begin{cases} 0 &\text{if } \|\lambda\|< r, \\ e^{\|\lambda\|-r}-\sum_{p=0}^q \frac{(\|\lambda\|-r)^p}{p!} & \text{if }\|\lambda\|\geq r,\end{cases}$ where $r>0$ and $q\geq 1$.
\end{enumerate}
Then, $\alpha_1,\alpha_2,\alpha_3$ are all adequate augmentations for the error function $E(\mathbf{z},\mathbf{y})=\|\mathbf{z}-\mathbf{y}\|^2$ under the assumptions of Lemma \ref{lemma-error-bound-G}.
\end{corollary}

\begin{proof}
One can see that $\alpha_i$ is differentiable and that $\nabla \alpha_i$ is locally Lipschitz for $i=1,2,3$. A straightforward computation shows that, for $i=1,2,3$, $\lambda^T \nabla \alpha_i (\lambda) > \Theta_\rho(\|\lambda\|^{h(G)+1}+1)$ if $\|\lambda\|\gg 1$. So, by Lemma \ref{lemma-error-bound-G}, we have that, for $i=1,2,3$,
\[
\lambda^T \nabla_\lambda \left(E(F(\mathbf{x},\lambda),g(\mathbf{x}))+\alpha_i(\lambda)\right)\geq \lambda^T \nabla \alpha_i (\lambda) - \|\lambda\|\cdot\|\nabla_\lambda \left(E(F(\mathbf{x},\lambda),g(\mathbf{x}))\right)\| >0 \text{ for } \|\lambda\|\gg 1.
\]
Thus, $\alpha_1,\alpha_2,\alpha_3$ are all adequate augmentations for the error function $E(\mathbf{z},\mathbf{y})=\|\mathbf{z}-\mathbf{y}\|^2$.
\end{proof}

\subsection{Augmented back-propagation algorithm} Next, we use Theorem \ref{thm-sgd} to establish a convergent augmented back-propagation algorithm for acyclic neural networks.

\begin{proposition}\label{prop-AA-BP-G}
Let $F:\R^{n}\times \R^N\rightarrow \R^{m}$ be the function represented by the acyclic neural network $G$. Fix a $0 <\rho\leq \infty$. Assume that:
\begin{enumerate}
	\item The functions $\{\s^i_j\}$ are uniformly $C^2$-bounded, that is, there is a $M>0$ such that $\s^i_j(t)\leq M$, $(\s^i_j)'(t)\leq M$ and $(\s^i_j)''(t)\leq M$ for all $i=1,\dots, l$, $j=1,\dots, n_i$ and $t\in\R$.
	\item $\mu$ is a probability measure on $B^{n}_\rho(\mathbf{0})$. $\{\mathbf{x}(k)\}_{k=0}^\infty$ is a sequence of independent random variables with identical probability distribution taking values in $B^{n}_\rho(\mathbf{0})$, where the probability distribution of each $\mathbf{x}(k)$ is $\mu$.
	\item $g:B^{n}_\rho(\mathbf{0})\rightarrow \R^{m}$ is a bounded function, that is, there is a $G>0$ such that $\|g(\mathbf{x})\|\leq G$ for all $\mathbf{x} \in B^{n}_\rho(\mathbf{0})$. And $\mathbf{y}(k)=g(\mathbf{x}(k))$ for $k=0,1,2,\dots$
	\item $E:\R^{m}\times\R^{m}\rightarrow \R$ is an error function that is $C^2$-bounded with respect to the first $m$ variables in $B^{m}_r \times B^{m}_G$ for every $r>0$. That is, there is a $\Gamma_r>0$ satisfying 
	\[
	|E(\mathbf{z},\mathbf{y})| + \|\nabla_{\mathbf{z}}E(\mathbf{z},\mathbf{y})\| + \sqrt{\Tr ((H_{\mathbf{z}}E(\mathbf{z},\mathbf{y}))^T H_{\mathbf{z}}E(\mathbf{z},\mathbf{y}))} \leq \Gamma_r\]
	for all $\mathbf{z}\in B^{m}_r$ and $\mathbf{y}\in B^{m}_G$, where $H_{\mathbf{z}}E(\mathbf{z},\mathbf{y})$ is the Hessian matrix of $E(\mathbf{z},\mathbf{y})$ with respect to $\mathbf{z}$.
	\item $\alpha:\R^N\rightarrow \R$ is an adequate augmentation with respect to $E,F,g$. Fix an $R_0>0$ such that 
	\[
	\lambda^T \nabla_\lambda \left(E(F(\mathbf{x},\lambda),g(\mathbf{x}))+\alpha(\lambda)\right)\geq 0
	\]
	whenever $\|\mathbf{x}\|\leq \rho$ and $\|\lambda\| \geq R_0$.
	\item $\{\eta_k\}_{k=0}^\infty$ is a sequence of positive numbers satisfying $\sum_{k=0}^\infty\eta_k =\infty$ and $\sum_{k=0}^\infty \eta_k^2 <\infty$. Set $H = \sup\{\eta_k~|~k\geq 0\}$.
	\item 
	\begin{itemize}
		\item $\lambda_0: \mathfrak{E} \rightarrow \R$ is any fixed function viewed as an element of $\R^N$, 
		\item $R_1 = \max\{\sqrt{\|\lambda_0\|^2 + \sum_{k=0}^\infty \eta_k^2},\sqrt{R_0^2 +2HR_0 +\sum_{k=0}^\infty \eta_k^2}\}$,
		\item $\Phi =\sup \{\|\nabla_\lambda \left(E(F(\mathbf{x},\lambda),g(\mathbf{x}))+\alpha(\lambda)\right)\|~|~ (\mathbf{x},\lambda) \in  B^{n}_\rho(\mathbf{0}) \times B^N_{R_1}(\mathbf{0})\}<\infty$,
		\item Fix a positive number $\vf$ satisfying $\vf\geq \Phi$.
	\end{itemize}
\end{enumerate}
Define a sequence of random functions $\{\lambda_k:\mathfrak{E} \rightarrow \R\}_{k=0}^\infty$ by the following:
\begin{itemize}
	\item Forward propagation given inductively by \eqref{eq-forward-bp-G};
  \item Backward propagation given inductively by \eqref{eq-backward-bp-G};
  \item \textbf{Augmented weight update} given by
	\begin{equation}\label{eq-weight-update-aug-G}
	\lm_{k+1}(e) = \lm_{k}(e) - \frac{\eta_k}{\vf} \left(\frac{\partial}{\partial \lm(e)} (E(F(\mathbf{x}(k),\lambda),\mathbf{y}(k))+\alpha(\lambda))\right)\big{|}_{\lambda=\lambda_k}.
	\end{equation}
\end{itemize}
Define the mean error function $\mathcal{E}: \R^N\rightarrow \R$ by $\mathcal{E}(\lambda) = \int_{B^{n}_\rho(\mathbf{0})}E(F(\ast,\lambda),g(\ast)) d\mu$. Then 
\begin{itemize}
	\item $\{\mathcal{E}(\lambda_k)+\alpha(\lambda_k)\}$ converges almost surely to a finite number,
	\item $\{\nabla\mathcal{E}(\lambda_k)+\nabla\alpha(\lambda_k)\}$ converges almost surely to $\mathbf{0}$,
	\item any limit point of $\{\lambda_k\}$ is almost surely a stationary point of the function $\mathcal{E} + \alpha$.
\end{itemize}
\end{proposition}

\begin{proof}
It is straightforward to check that Proposition \ref{prop-AA-BP-G} follows from the application of Theorem \ref{thm-sgd} to 
\begin{itemize}
	\item the function $f:\R^N\times B^{n}_\rho \rightarrow \R$ given by $f(\lambda,\mathbf{x}) := E(F(\mathbf{x},\lambda),g(\mathbf{x})) + \alpha(\lambda)$,
	\item the inputs $\{\mathbf{x}(k)\}_{k=0}^\infty$,
	\item the sequence $\{\eta_k\}_{k=0}^\infty$.
\end{itemize}
\end{proof}

\end{document}